\documentclass{article}          %pdflatex

\usepackage{amsmath}
\usepackage{graphicx}
\usepackage{color}
\usepackage{multirow}
\usepackage{setspace}
\usepackage{latexsym}
\usepackage{amsfonts}
\usepackage{amsthm}

\usepackage{caption}
\usepackage{subcaption}
\usepackage[a4paper, margin=3cm]{geometry}
\usepackage[numbers,sort&compress]{natbib}

\usepackage[colorlinks,
            linkcolor=blue,
            anchorcolor=blue,
            citecolor=blue
            ]{hyperref}

\def\e{\boldsymbol{e}}

\def\x{\boldsymbol{x}}
\def\y{\boldsymbol{y}}
\def\z{\boldsymbol{z}}
\def\f{\boldsymbol{f}}

\def\r{\boldsymbol{r}}

\def\v{\boldsymbol{v}}
\def\h{\boldsymbol{h}}

\def\n{\boldsymbol{n}}

\def\u{\boldsymbol{u}}

\def\z{\boldsymbol{z}}

\def\1{\boldsymbol{1}}

\begin{document}

\title{Controllability of Keplerian Motion with Low-Thrust Control Systems\thanks{This work was partially supported by a public grant overseen by the French National Research Agency (ANR) as part of the $\ll$ Investissement d'Avenir $\gg$ program, through the ``iCODE Institute Project'' funded by IDEX Paris-Saclay, ANR-11-IDEX-0003-02.}}
%\headlinetitle{Controllability of Keplerian motion}

\author{
Zheng Chen\thanks{D\'epartement de Math\'ematiques \& CNRS, University of Paris-Sud, Bat. 425, Facult\'e des Sciences d'Orsay, F-91405, Orsay CEDEX, France, \underline{zheng.chen@math.u-psud.fr.}}
%{\it Northwestern Polytechnical University, Xi'an, Shaanxi, China, 710072}\\
%{\it Universit\'e Paris-Sud $\&$ CNRS, Orsay, Paris, 91405, France}\\
%\\
\ and
%{\it Universit\'e de Dijon $\&$ CNRS, Dijon, Bourgogne, 21078, France}\\
%\\
 Yacine Chitour\thanks{Laboratoire des Signaux et Syst\`emes, Universit\'e Paris-Sud, CNRS, and CentraleSupelec, Gif-sur-Yvette, France. \underline{yacine.chitour@lss.supelec.fr.}},
%{\it Universit\'e Paris-Sud $\&$ L2S(CNRS), Orsay, Paris, 91405, France}
%\thanks{PhD Candidate, Mathematics Department, \underline{chenzh@math.u-psud.fr.}}\\
%{\normalsize\itshape
%   Universit'e Paris-Sud, Orsay, Paris, 91405, France}\\
}

\maketitle{}
\newtheorem{property}{Property}%[section]
\newtheorem{definition}{Definition}%[section]
\newtheorem{proposition}{Proposition}%[section]
\newtheorem{problem}{Problem}%[section]
\newtheorem{remark}{Remark}%[section]
\newtheorem{assumption}{Assumption}%[section]
\newtheorem{hypothesis}{Hypothesis}%[section]
\newtheorem{conjecture}{Conjecture}%[section]
\newtheorem{theorem}{Theorem}%[section]
\newtheorem{corollary}{Corollary}%[section]
\newtheorem{lemma}{Lemma}%[section]

\begin{abstract}
In this paper, we present the controllability properties of Keplerian motion controlled by low-thrust control systems. The low-thrust control system, compared with high or even impulsive control system, provide a fuel-efficient means to control the Keplerian motion of a satellite in restricted two-body problem. We obtain that, for any positive value of maximum thrust, the motion is controllable for orbital transfer problems. For two other typical problems: de-orbit problem and orbital insertion problem, which have state constraints, the motion is controllable if and only if the maximum thrust is bigger than a limiting value. Finally, two numerical examples are given to show the numerical method to compute the limiting value.
\end{abstract}

%\keywords{Keplerian motion, Low thrust, Controllability}

%\classification{49K15,70Q05}

%\researchsupported{This work was partially supported by a public grant overseen by the French National Research Agency (ANR) as part of the $\ll$ Investissement d'Avenir $\gg$ program, through the ``iCODE Institute Project'' funded by IDEX Paris-Saclay, ANR-11-IDEX-0003-02.}

%% If there is an additional footnote on page 1, place ``\makethankshere'' subsequent
%% to that footnote and use the class option ``nothanks''.

%\acknowledgments{Please insert acknowledgments of the assistance of colleagues or similar notes of appreciation here.}

%-------------------------------------------------------------------------%
% Please do not alter the following line.                                 %
%\firstpage{1}
%-------------------------------------------------------------------------%

%-------------------------------------------------------------------------%
% To include a table of chapters, write                                   %
   \tableofcontents
%-------------------------------------------------------------------------%

\section{Introduction}

In classical mechanics, the determination of the motion of two celestial bodies, which interact only with each other, is the typical two-body problem. If one body is light enough, the uncontrolled motion of the light body around a heavy body is a restricted two-body problem, and the motion is well-known as Keplerian motion. A common example is the artificial bodies, i.e., spacecrafts or satellites, moving around the Earth. Once the atmospheric effects are negligible and the Earth overwhelmingly dominates the gravitational influence, a satellite moves stably on a periodic orbit if the mechanical energy of the satellite is negative. As an increasing number of artificial satellites have been launched into space around the Earth or even into deeper space since the mid of last century, an important problem arises in astronautics, that is to control the Keplerian motion of a satellite to transfer between different orbits to achieve desired mission requirements.

The control of an artificial satellite is generally performed by system propulsion,
expelling mass in a high speed to generate an opposite reaction force according to
Newton's third law of motion. Up to now, there are already several types of propulsion
systems available, including chemical propulsion systems and electric propulsion
systems. Though chemical propulsion systems are able to provide much higher thrust,
electric propulsion systems have the the potential for a much higher specific impulse
than is available from chemical ones, resulting in a lower fuel consumption and thus a
longer satellite lifetime for a given propellant mass. On the one hand, the electric
propulsion systems provide a fuel-efficient means to control the motion of a
satellite; On the other hand, the fact that the possible maximum thrust, which the
electric propulsion systems can provide, is very low results that the transfer time is
exponentially long. Hence, the optimization of transfer time has be studied in
Refs.\ \cite{Bonnard:06,Bonnard:05,Caillau:01}. In addition, the strong-local optimality of such problems was studied in Ref.\ \cite{Caillau:15}.

Though the low-thrust control systems provide an fuel-efficient means to control the
motion of a satellite, the problem that whether or not it has the ability to move a
satellite from one point to another one arises. This is actually a controllability
property, which is a prerequisite to analyze  mission feasibility during designing a
space mission or designing an optimal trajectory. Restricting the mechanical energy of
a satellite into negative region without any other state constraints, the controlled
motion is called orbital transfer problem (OTP), and the controllability for OTP was
derived in Ref.\ \cite{Bonnard:05,Caillau:01} to show that there exists admissible
controlled trajectories for every OTP if the maximum thrust is positive. In the
current paper, the controllability for OTP is established using alternative techniques from geometric control, (cf. Refs.\ \cite{Sontag:98,Jean:14}). Taking into account the state constraint that the radius of a satellite is larger than the radius of the surface of the atmosphere around the Earth, the orbital insertion problem (OIP) and de-orbit problem (DOP) are defined in this paper. Some controllability properties for OIP and DOP are  then addressed and we show that there exist admissible controlled trajectories for OIPs and DOPs if and only if the maximum thrust is bigger than a specific value (depending on the initial point or final point).

The organization of the paper is the following. In Section \ref{SE:Notation}, we recall the basic properties of the dynamics of the motion of a satellite around the Earth, and basic notations and definitions are given which are crucial for analysis of controllability properties. In Section \ref{SE:Controllability}, the controllability for OTPs is represented firstly by using geometric control technology in Ref.\ \cite{Jean:14}. Then, the controllability properties for OIPs and DOPs are derived in Subsections \ref{SE:P-_P+} and \ref{SE:P+_P-}, respectively. In Section \ref{SE:Numerical}, two numerical examples are given to show the development in this paper. Finally, a conclusion is given in Section \ref{SE:Conclusion}.

\section{Notations and definitions}\label{SE:Notation}

\subsection{Dynamics}

Consider a satellite as a mass point moving around the Earth, its state in a geocentric inertial cartesian coordinate (GICC), illustrated by Figure \ref{Fig:ECICS}, consists of its position vector $\r \in \mathbb{R}^3\backslash\{0\}$, velocity vector $\v \in \mathbb{R}^3$, and mass $m\in\mathbb{R}^*_+$. Then, the dynamics for the movement of the satellite for positive times can be written as:
\begin{eqnarray}
\Sigma_{\text{sat}}:
\begin{cases}
\dot{\r}(t) = \v(t),\\
\dot{\v}(t) = -\frac{\mu}{\parallel \r(t) \parallel^3}\r(t) + \frac{\boldsymbol{\tau}(t)}{m(t)},\\
\dot{m}(t) = -\beta \parallel \boldsymbol{\tau}(t) \parallel,
\end{cases}
\label{EQ:Sigma}
\end{eqnarray}
where  $\mu > 0$ is the gravitational constant, $\beta > 0$ is a scalar constant determined by the specific impulse of the low-thrust control system equipped on the satellite, $\parallel \cdot \parallel$ denotes the Euclidean norm and the thrust (or control) vector $\boldsymbol{\tau}\in\mathbb{R}^3$ takes values in the admissible set
\begin{equation}
\mathcal{T}(\tau_{\mathrm{max}} )=\big\{\boldsymbol{\tau} \in \mathbb{R}^3\ \arrowvert\  \parallel \boldsymbol{\tau} \parallel \leq \tau_{\mathrm{max}}   \big\},
\label{EQ:thrust_admissible}
\end{equation}
where $\tau_{\mathrm{max}}$ is a positive constant. 
\begin{figure}[!h]
 \centering\includegraphics[width=1.5in]{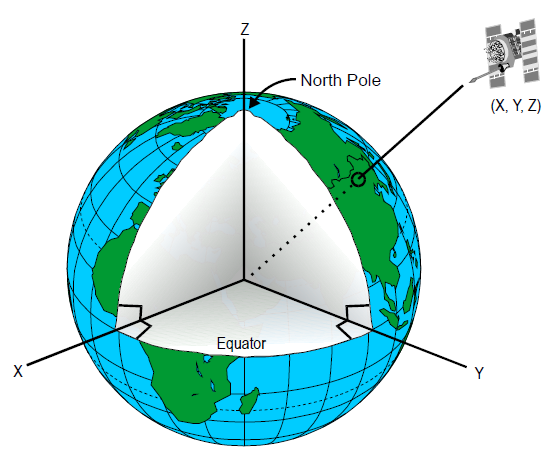}
 \caption[]{Geocentric inertial cartesian coordinate.}
 \label{Fig:ECICS}
\end{figure}

If $\mathcal{X}=\mathbb{R}^3\backslash\{0\}\times\mathbb{R}^3$ and $\x = (\r,\v)$, we define two vector fields $\f_0,\f_1$ on $\mathcal{X}$ by 
 \begin{eqnarray}
&&\f_0: \mathcal{X}\rightarrow \mathbb{R}^6,\ \f_0(\x) = \left(
\begin{array}{c}
\v\\
-\frac{\mu}{\parallel \r \parallel^3}\r
\end{array}\right),\\
&&\f_1: \mathcal{X}\rightarrow \mathbb{R}^{6\times 3},\ \f_1(\x) = \left(
\begin{array}{c}
\boldsymbol{0}\\
I_3
\end{array}\right),
\end{eqnarray}
with $\mathbb{R}^{6\times 3}$ denotes the set of $6\times 3$ matrices wih real entries and $I_3$  is the identity matrix of $\mathbb{R}^3$. 
Let $\mathcal{B}_\varepsilon$ be the closed ball in
$\mathbb{R}^3$ centered at the origin and of radius $\varepsilon>0$. For every $\varepsilon > 0$, we consider the control-affine system $\Sigma_{\varepsilon}$ given by 
\begin{eqnarray}
\Sigma_{\varepsilon}: \dot{\x}(t) = \f_0(\x(t)) + \f_1(\x(t)) \u(t),
\label{EQ:Sigma_varepsilon}
\end{eqnarray}
where the control vector $\u\in\mathbb{R}^3$ takes values in $\mathcal{B}_\varepsilon$.
We will use in this paper the vector field point of view of Refs.\ \cite{Sussmann:98,Sussmann:00,Jurdjevic:97}. For every point $\x\in\mathcal{X}$ and every $\u\in \mathcal{B}_{\varepsilon}$, we denote by
\begin{eqnarray}
\f:\mathcal{X}\times\mathcal{B}_{\varepsilon} \rightarrow T_{\x}\mathcal{X},\ (\x,\u)\mapsto \f(\x,\u)= \f_0(\x) + \f_1(\x)\u,
\end{eqnarray}
where $\f_0$ and $\f_1$ are referred to as the {\it drift vector field} and the {\it control vector field}, respectively. Note that trajectories of $\Sigma_{\varepsilon}$ starting at any $\x_0\in\mathcal{X}$
and measurable $\u:\mathbb{R}_+\rightarrow \mathcal{B}_\varepsilon$ are well-defined on an open interval of $\mathbb{R}_+$ containing $0$, which depends in general on $\x_0$ and $\u(\cdot)$.

\subsection{Study of the drift vector field  in $\mathcal{X}$}
In this paragraph, we recall the main properties of the drift vector field $\f_0$.
For every $\x\in\mathcal{X}$, we use $\gamma_{\x}$ to denote the restriction to 
$\mathbb{R}_+$ of the maximal trajectory of $\f_0$ starting at $\x$, i.e. $\gamma_{\x}$ is defined on some interval $[0,t_f(\x))$ where $t_f(\x)\leq \infty$. Then the follwing holds true.

\begin{property}[First integrals \cite{Bismut:11,Curtis:05}]
For every $\x\in\mathcal{X}$, if $\gamma_{\x}(t)=(\tilde{\r}(t),\tilde{\v}(t))$ on $[0,t_f(\x))$, the quantities 
\begin{eqnarray}
{\h}(t) &=& \tilde{\r}(t)\times\tilde{\v}(t),\label{EQ:angular_momentum}
\end{eqnarray}
\begin{eqnarray}
{\boldsymbol{L}}(t) &=& {\tilde{\v}(t)\times{\h}}- {\mu} \frac{\tilde{\r}(t)}{\parallel \tilde{\r}(t)\parallel},\label{EQ:Laplace_integral}
\end{eqnarray}
\begin{eqnarray}
E(t) &=& \frac{\parallel \tilde\v(t) \parallel^2 }{2} - \frac{\mu}{\parallel \tilde\r(t) \parallel},
\label{EQ:energy}
\end{eqnarray}
are constant along $\gamma_{\x}$ and the corresponding constant values are the angular momentum vector ${\h}\in\mathbb{R}^3$, the Laplace vector ${\boldsymbol{L}}\in\mathbb{R}^3$  and the mechanical energy of a unit mass $E\in\mathbb{R}$, which is the sum of the relative kinetic energy $\parallel \tilde{\v}(t)\parallel^2/2$ and the potential energy $-\mu/\parallel \tilde{\r}(t)\parallel$.
\label{DE:First_integrals}
\end{property}

\noindent As a consequence of Eq.(\ref{EQ:angular_momentum}) and Eq.(\ref{EQ:Laplace_integral}), we have the following two properties.
\begin{property}[Straight line \cite{Bismut:11,Curtis:05}]
Let $\x\in\mathcal{X}$ with $\h=0$, i.e. $\r$ and $\v$ are colinear. Then the trajectory $\gamma_{\x}$ is a straight line and $t_f(\x)$ is either finite or infinite.
\end{property}
\begin{property}[Conic section \cite{Bismut:11,Curtis:05}]
Let $\x\in\mathcal{X}$ with $\h \neq 0$ i.e. $\r$ and $\v$ are not colinear. Then, $t_f(\x)=\mathbb{R}^+$ and the trajectory $\gamma_{\x}$ is a periodic trajectory with locus defining a conic section lying in a two-dimensional plane perpendicular to $\h$ called the orbital plane.
\end{property}

\noindent Let
\begin{eqnarray}
\tilde{\mathcal{X}} = \{(\r,\v)\in\mathcal{X}\ \arrowvert\  \r\times \v=\h \neq 0\}.
\end{eqnarray}
%For every $\x\in\tilde{\mathcal{X}}$, the trajectory $\gamma(t,\r,\v)$ on $\mathbb{R}^+$ is called an orbit, and we say the two-dimensional plane, on which the orbit $\gamma(t,\r,\v)$ on $\mathbb{R}^+$ lies, is the orbital plane. Given every point $(\r,\v)\in\tilde{\mathcal{X}}$, let us define the following two constants:
%\begin{eqnarray}
%a: \tilde{\mathcal{X}}\rightarrow \mathbb{R},\ a(\r,\v) &=& - \frac{\mu}{2 E},\label{EQ:a}\\
%e: \tilde{\mathcal{X}}\rightarrow \mathbb{R},\ e(\r,\v) &=& \parallel \boldsymbol{L}\parallel /\mu.\label{EQ:e}
%\end{eqnarray}
%If $\boldsymbol{L} \neq 0$, let $(\tilde{\r}(t),\tilde{\v}(t))=\gamma(t,\r,\v)$ on $\mathbb{R}^+$, we then have
%\begin{eqnarray}
%\parallel \tilde{\r}(t) \parallel = \frac{a(\r,\v) (1 - e(\r,\v)^2)}{1 + e(\r,\v) \cos\theta(t,\r,\v)},
%\label{EQ:r}
%\end{eqnarray}
%where $\theta(t,\r,\v): \mathbb{R}^+ \times \tilde{\mathcal{X}}\rightarrow\mathbb{S}$ is the angle between $\tilde{\r}(t)$ and $\boldsymbol{L}$, defined by
%\begin{eqnarray}
%\theta(t,\r,\v)= \begin{cases}
% \cos^{-1}\left(\frac{\boldsymbol{L}^T\cdot \tilde{\r}(t)}{\parallel \tilde{\r}(t)\parallel \parallel \boldsymbol{L}\parallel}\right),\ (\tilde{\v}(t)^T\cdot \tilde{\r}(t)> 0),\\
% 2\pi - \cos^{-1}\left(\frac{\boldsymbol{L}^T\cdot \tilde{\v}(t)}{\parallel \tilde{\r}(t)\parallel \parallel \boldsymbol{L}\parallel}\right),\ (\tilde{\v}(t)^T\cdot \tilde{\r}(t) < 0).
% \end{cases}
% \nonumber
%\end{eqnarray}
Define on $\mathcal{X}$ the function $e:\x\mapsto  \parallel \boldsymbol{L}\parallel /\mu$. Along every trajectory of $\f_0$ starting at $\x\in\tilde{\mathcal{X}}$, one gets, after multiplying Eq.(\ref{EQ:Laplace_integral}) by $\tilde{\r}(t)$, that 
\begin{eqnarray}
\parallel \tilde{\r}(t) \parallel = \frac{ \parallel \h\parallel^2}{\mu(1 + e(\x) \cos\theta(t,\x))},
\label{EQ:r}
\end{eqnarray}
where the angle $\theta(t,\x)$ is defined by $\cos\theta(t,\x)=\frac{\boldsymbol{L}^T\cdot \tilde{\v}(t)}{\parallel \tilde{\r}(t)\parallel \parallel \boldsymbol{L}\parallel}$. Note that the previous formula holds true if $ \boldsymbol{L}=0$ since in that case  $e(\x) \cos\theta(t,\x)$ is equal to zero and the orbit is a circle. 

%\begin{definition}[Semi-major axis and eccentricity]
%Given every point $(\r,\v)\in\tilde{\mathcal{X}}$, we say that $a(\r,\v)$ and $e(\r,\v)$ are the semi-major axis and eccentricity of the orbit $\gamma(t,\r,\v)$  on $\mathbb{R}^+$, respectively.
%\end{definition}

\begin{figure}[!h]
\centering \includegraphics[trim=3.5cm 12.7cm 3.5cm 5cm, clip=true, width=3.5in, angle=0]{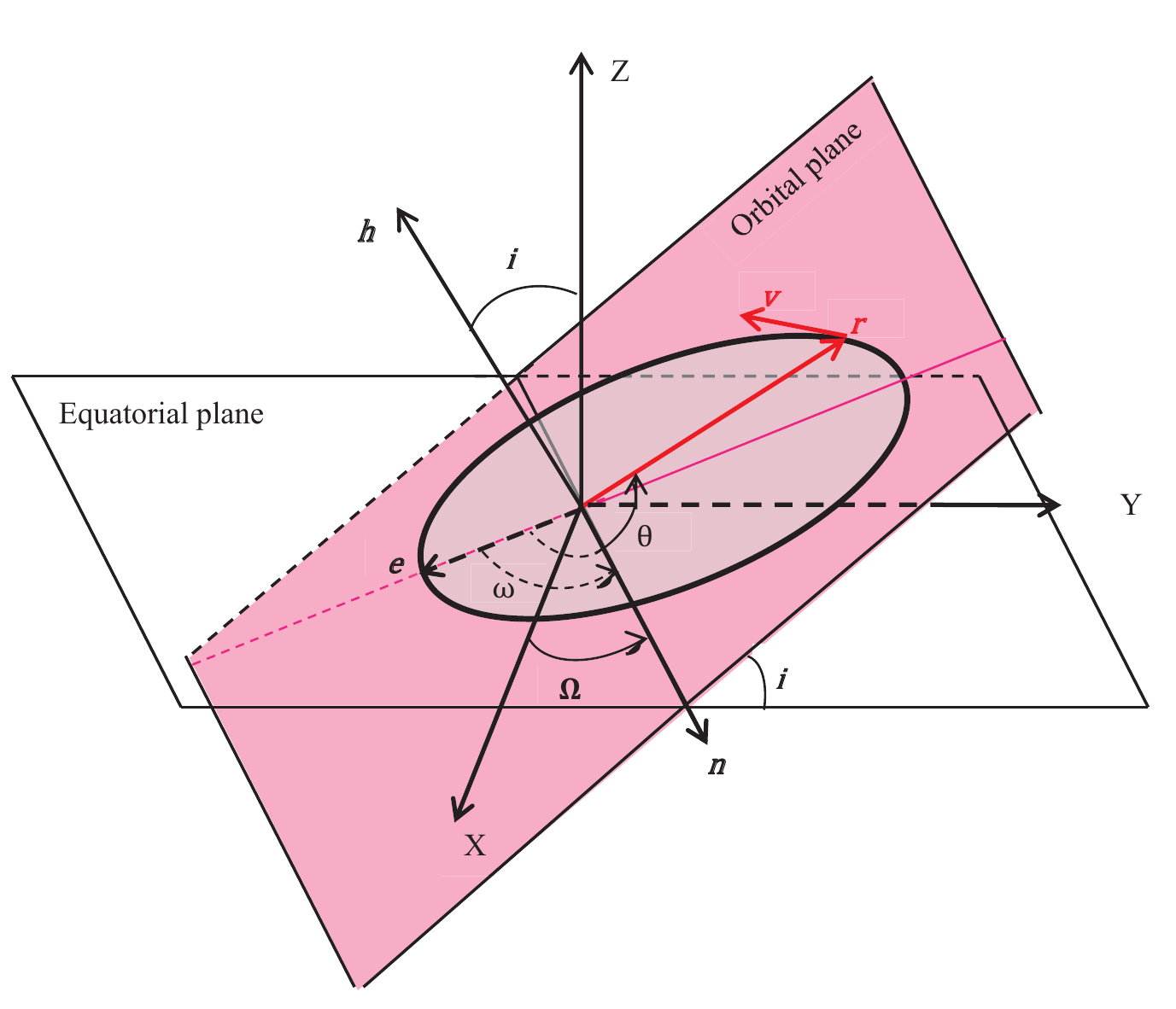}
\caption{The orientation of a 2-dimensional orbital plane in GICC and the geometric shape and orientation of an elliptic orbit on the orbital plane.}
\label{Fig:orbparam}
\end{figure}

Notice from Eq.(\ref{EQ:r}) that an orbit $(\tilde{\r}(t),\tilde{\v}(t))=\gamma_{\x}(t)$ with $\x\in\tilde{\mathcal{X}}$ on $\mathbb{R}^+$ is a parabola if $e(\x) = 1$ and a hyperbola if $e(\x)> 1$. Put a satellite on a parabolic or a hyperbolic orbit without no control, then it can escape to infinity
$\underset{t\rightarrow +\infty}{\text{lim}} \tilde{\r}(t) = +\infty$. Thus, parabolic and hyperbolic orbits are generally used for a satellite to escape from the gravitational attraction of Earth. 
For every point $\x\in\tilde{\mathcal{X}}$, if $0 \leq e(\x) < 1$, the orbit $\gamma_{\x}(t)$ on $\mathbb{R}^+$ is an ellipse, whose orientation is illustrated in Figure \ref{Fig:orbparam}. 
Moreover, it is easy to deduce the following characterization of elliptic orbits. 

\begin{property}
Given every point $\x\in\tilde{\mathcal{X}}$, the mechanical energy $E$ is negative if and only if  $e(\x) < 1$.
\end{property}
\noindent Thus, let us define the set
\begin{eqnarray}
\mathcal{P} = \{\x\in \tilde{\mathcal{X}}\ \arrowvert\ E<0 \},
\end{eqnarray}
then for every point $(\r,\v)\in\mathcal{P}$, the associated orbit $\gamma_{\x}$ on $\mathbb{R}^+$ is periodic and the set $\mathcal{P}$ is called the periodic region in $\mathcal{X}$.

\begin{definition}[Smallest period $t_p$]
Given every point $\x\in\mathcal{P}$, we denote by
 $$t_p:\mathcal{P}\rightarrow \mathbb{R},\ \x\mapsto t_p(\x),$$
 the smallest period of the orbit $\gamma_{\x}$ on $\mathbb{R}^+$.
\end{definition}

\noindent According to Eq.(\ref{EQ:r}), for every point $\x\in\mathcal{P}$, if $e(\x)\neq 0$, the associated orbit $\gamma_{\x}$ on $[0,t_p(\x)]$ has its perigee point and apogee point at $\theta(t,\x) = 0$ and $\pi$, respectively. Thus, let 
\begin{eqnarray}
r_p:\mathcal{P}\rightarrow \mathbb{R},\ r_p(\x) &=& \frac{ \parallel \h\parallel^2}{\mu(1 + e(\x))},
\label{EQ:r_min}\\
r_a:\mathcal{P}\rightarrow \mathbb{R},\ r_a(\x) &=& \frac{ \parallel \h\parallel^2}{\mu(1 - e(\x) )},
\label{EQ:r_max}
\end{eqnarray}
we say $r_p(\x)$ and $r_a(\x)$ are the perigee and apogee distances of the orbit $\gamma_{\x}$ on $[0,t_p(\x)]$ if $e(\x)\neq 0$. Note that $r_{a}(\x) = r_{p}(\x)$ if and only if $e(\x) = 0$, which corresponds to a circular orbit. 

\begin{property}[Minimum radius and maximum radius]
Given every periodic orbit $(\tilde{\r}(t),\tilde{\v}(t))=\gamma_{\x}(t)$  on $[0,t_p(\x)]$ in $\mathcal{P}$, we have $r_p(\x) \leq \parallel \tilde{\r}(t)\parallel \leq r_a(\x)$ on $[0,t_p(\x)]$. Thus, the perigee distance $r_p(\x)$ and apogee distance $r_a(\x)$ are the minimum radius and maximum radius of the orbit $(\tilde{\r}(t),\tilde{\v}(t))$ on $[0,t_p(\x)]$.
\end{property}

\subsection{Admissible controlled trajectory of $\Sigma_{\mathrm{sat}}$}

For every initial point $\y_i=(\x_i,m_i)\in\tilde{\mathcal{X}}\times\mathbb{R}^*_+$ and measurable control function $\boldsymbol{\tau}(\cdot)$ taking values in $\mathcal{T}(\tau_{\mathrm{max}})$ with $\tau_{\mathrm{max}}>0$, let $\tilde{t}_f\in\mathbb{R}^+$ be the maximum time such that the corresponding trajectory $\Gamma(t,\boldsymbol{\tau},\y_i)$ of $\Sigma_{\mathrm{sat}}$ lies in $\mathcal{X}\times\mathbb{R}^*_+$, i.e., if $\Gamma(t,\boldsymbol{\tau},\y_i)=(\x(t),m(t))$, then $\x(t)\in\tilde{\mathcal{X}}$ and $m(t) > 0$ on $[0,\tilde{t}_f)$. We use ${\mathcal{I}}_{\Gamma}$ to denote $[0,\tilde{t}_f)$.
We say $\Gamma(t,\boldsymbol{\tau},\y_i,m_i)$ on ${\mathcal{I}}_{\Gamma}$ is the controlled trajectory of $\Sigma_{\mathrm{sat}}$ starting from $\y_i$ and associated with $\boldsymbol{\tau}(\cdot)$.

\begin{remark}
For every point $\y_i=(\x_i,m_i)\in\mathcal{P}\times \mathbb{R}_+^\ast$, let $(\x(t),m(t))=\Gamma(t,\boldsymbol{0},\y_i)$ on $\mathcal{I}_{\Gamma}$, we have that $\x(t) = \gamma_{\x_i}(t)$ and $m(t) = m_i$ for every $t\geq 0$, i.e., $\tilde{t}_f = \infty$.
\end{remark}

\begin{definition}[Controlled Keplerian motion]
Given every initial point $\y_i=(\x_i,m_i)\in\mathcal{P}\times\mathbb{R^*_+}$ and measurable control function $\boldsymbol{\tau}(\cdot)$ taking values in $\mathcal{T}(\tau_{\mathrm{max}})$ with $\tau_{\mathrm{max}}>0$, the corresponding trajectory $\Gamma(t,\boldsymbol{\tau}(\cdot),\y_i,m_i)$ of $\Sigma_{\mathrm{sat}}$ is called a controlled Keplerian motion.
\label{DE:Controlled_Keplerian_Motion}
\end{definition}

Let $r_c > 0$ and $M_0>0$ denote the radius of the surface of atmosphere around the Earth and the mass of a satellite without any fuel, respectively, then given every point $(\x,m)\in\mathcal{P}\times\mathbb{R^*_+}$ on the trajectories of Keplerian motions, it is required that $\parallel \r \parallel > r_c$ and $m > M_0$.  
\begin{definition}[Admissible region]
We define the set
\begin{eqnarray}
\mathcal{A} = \big\{\x=(\r,\v)\in\mathcal{P}\ \arrowvert\ \parallel \r \parallel > r_c \big\},
\label{EQ:admissible_region_A}
\end{eqnarray}
the admissible region in $\mathcal{P}$ for Keplerian motion and/or controlled Keplerian motion.
\end{definition}

\begin{definition}[Admissible controlled trajectory]
Given every $M_0 > 0$, we say the controlled trajectory $(\x(t),m(t))=\Gamma(t,\boldsymbol{\tau},\x_i,m_i)$ of  $\Sigma_{\mathrm{sat}}$ on some finite intervals $[0,t_f]\subset{\mathcal{I}}_{\Gamma}$ with initial condition $(\x_i,m_i)\in\mathcal{A}\times\mathbb{R}^*_+$ is an admissible controlled trajectory if $(\x(t))\in\mathcal{A}$ and $m(t) \geq M_0$ for $t\in[0,t_f]$.
\end{definition}  
\noindent For every time interval $[0,t_f]\subset\mathcal{I}_{\Gamma}$, since $\dot{m}(t)\leq 0$, it follows $m(t_f) \leq m(t)$ on $[0,t_f]$. Thus, the inequality $m(t)\geq M_0$ can be ensured by $m(t_f) \geq M_0$.

\subsection{Controlled problems in $\mathcal{A}$}

For $\x\in\mathcal{A}$, let $(\tilde{\r}(t),\tilde{\v}(t))=\gamma_{\x}(t)$ on $\mathbb{R}^+$, we have that the inequality $\parallel \tilde{\r}(t)\parallel > r_c$ is satisfied on $\mathbb{R}^+$ if $r_p(\x) > r_c$. Thus, we define the set:
\begin{eqnarray}
\mathcal{P}^{+} &=& \{(\r,\v)\in \mathcal{P}:  r_{p}(\x) > r_c\}.
\end{eqnarray}
It is immediate to see that the periodic uncontrolled trajectory $\gamma(t,\x)$ starting at any $\x\in\mathcal{P}^{+}$ remains in $\mathcal{P}^{+}$.

Let
\begin{eqnarray}
\mathcal{P}^{-} &=& \{\x=(\r,\v)\in \mathcal{P}\ \arrowvert\ \parallel {\r} \parallel > r_c,\ r_{p}(\x) < r_c< r_{a}(\x)\}.
\end{eqnarray}
Then, for every point $\x\in\mathcal{P}^-$, there exists an interval $[t_1,t_2]\in[0,t_p(\x)]$ such that $\parallel \tilde{\r}(t)\parallel \leq r_c$ for $t\in[t_1,t_2]$. Thus, placing a satellite on a point $\x\in\mathcal{P}^-$, it can move out of the admissible region $\mathcal{A}$.
\begin{definition}[Stable periodic region $\mathcal{P}^+$ and unstable periodic region $\mathcal{P}^-$ in $\mathcal{A}$]
We say that the two sets $\mathcal{P}^+$ and $\mathcal{P}^-$ are the stable and unstable periodic regions, respectively.
\end{definition}

\noindent All the satellites periodically moving around the Earth are located in the stable periodic region $\mathcal{P}^{+}$. In order to fulfill observation or other mission requirements, a satellite is controlled to move from one point $\x_i$ in $\mathcal{P}^{+}$ to another point $\x_f$ in $\mathcal{P}^+$ by its control system.

\begin{definition}[{\it Orbital Transfer Problem (OTP)}]
We say that the problem of controlling a satellite from a point $\x_i$ in $\mathcal{P}^{+}$ to another point $\x_f$ in $\mathcal{P}^+$ is the orbital transfer problem, see the first figure of Fig.~\ref{Fig:OTP}.
\label{DE:OTP}
\end{definition}

\noindent For a typical space mission, in order to place a satellite into a stable orbit in $\mathcal{P}^{+}$, a rocket is used to carry the satellite from the surface of the Earth to a point $\x_i$ in $\mathcal{P}^{-}$, at which the rocket and the satellite are separated. From this moment on, the satellite is controlled by its own control system to be inserted into a stable orbit in $\mathcal{P}^{+}$.
\begin{definition}[{\it Orbital Insertion Problem (OIP)}]
We say that the problem of controlling a satellite from an initial point $\x_i\in \mathcal{P}^{-}$ to a final point $\x_f\in\mathcal{P}^{+}$ is the orbit insertion problem, see the third figure of Fig.~\ref{Fig:OIP}.
\end{definition}

\noindent After a satellite in the stable region $\mathcal{P}^{+}$ finishes its mission, it should be decelerated to return to the unstable region $\mathcal{P}^{-}$. Then, the satellite will coast into atmosphere such that the aerodynamic pressure will act as a control to control the satellite to fly to landing sites.
\begin{definition}[{\it De-Orbit Problem (DOP)}]
We say that the problem of controlling a satellite from an initial point $\x_i\in \mathcal{P}^{+}$ to a final point $\x_f\in\mathcal{P}^{-}$ is the de-orbit problem, see the second figure of Fig.~\ref{Fig:DOP}.
\end{definition}

\begin{figure}[!h]%trim=0cm 4.5cm 0cm 4.5cm, clip=true, width=5.0in, angle=0]
    \centering
    \begin{subfigure}[b]{0.3\textwidth}
        \centering
        \includegraphics[trim=0.5cm 0.3cm 16.5cm 0.5cm, clip=true, width=1\textwidth, angle=0]{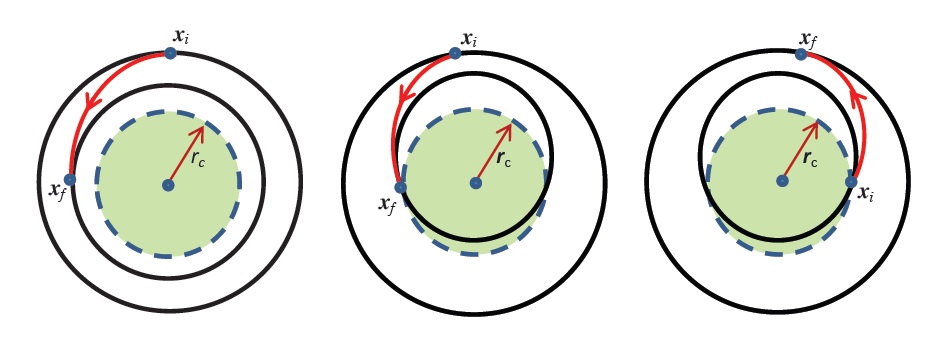}
        \caption{OTP}
        \label{Fig:OTP}
    \end{subfigure}
     \hfill
    \begin{subfigure}[b]{0.3\textwidth}
        \centering
        \includegraphics[trim=16.5cm 0.3cm 0.5cm 0.5cm, clip=true, width=1\textwidth, angle=0]{Transfering_Orbit.jpg}
        \caption{OIP}
        \label{Fig:OIP}
    \end{subfigure}
    \hfill
    \begin{subfigure}[b]{0.3\textwidth}
        \centering
        \includegraphics[trim=8.5cm 0.3cm 8.5cm 0.5cm, clip=true, width=1\textwidth, angle=0]{Transfering_Orbit.jpg}
        \caption{DOP}
        \label{Fig:DOP}
    \end{subfigure}
       \caption{OTP, OIP, and DOP.}
    \label{Fig:Transferring_Orbit}
\end{figure}

\section{Controllability}\label{SE:Controllability}

According to the definition for controlled Keplerian motion in {\it Definition \ref{DE:Controlled_Keplerian_Motion}}, the controllability of Keplerian motion deals with the existence of admissible controlled trajectories for OTP, OIP, and DOP.

\begin{definition}[Controllability for OTP]
We say that the system $\Sigma_{\mathrm{sat}}$ is controllable for OTP  if there exists $\tau_{\mathrm{max}}>0$ so that, for every initial mass $m_i > 0$ and every initial and final points $(\x_i,\x_f)\in(\mathcal{P}^{+})^2$, there exists a time $t_f\in{\mathcal{I}}_{\Gamma}$ and an admissible controlled trajectory $(\x(t),m(t))=\Gamma(t,\boldsymbol{\tau},\x_i,m_i)$ of $\Sigma_{\mathrm{sat}}$  on $[0,t_f]$ in $\mathcal{A}\times\mathbb{R}^*_+$ such that $\x(t_f) = \x_f$.
\end{definition}

\begin{definition}[Controllability for OIP and DOP]
We say that the system $\Sigma_{\mathrm{sat}}$ is controllable for OIP (DOP respectively) from any point 
$\x_i\in \mathcal{P}^{-}$ ($\x_i\in \mathcal{P}^{+}$ respectively) if for every initial mass $m_i > 0$
there exists $\tau_{\mathrm{max}}>0$ so that, for every final point $\x_f\in\mathcal{P}^{+}$ ($\x_f\in \mathcal{P}^{-}$ respectively), there exists a time $t_f\in{\mathcal{I}}_{\Gamma}$ and an admissible controlled trajectory $(\x(t),m(t))=\Gamma(t,\boldsymbol{\tau},\x_i,m_i)$ of $\Sigma_{\mathrm{sat}}$  on $[0,t_f]$ in $\mathcal{A}\times\mathbb{R}^*_+$ such that $\x(t_f) = \x_f$.
\end{definition}

%\begin{definition}[Controllability for OIPs]
%We say that the system $\Sigma_{\mathrm{sat}}$ is controllable for OIP if for every initial mass $m_i > 0$, every initial point $(\r_i,\v_i)\in\mathcal{P}^{-}$, and every final point $(\r_f,\v_f)\in\mathcal{P}^{+}$, there exists a finite time $t_f \in{\mathcal{I}}_{\Gamma}$, an $M_0> 0$, and an admissible controlled trajectory $(\r(t),\v(t),m(t))=\Gamma(t,\boldsymbol{\tau}(t),\r_i,\v_i,m_i)$ of $\Sigma_{\mathrm{sat}}$ on $[0,t_f]$  in $\mathcal{A}\times\mathbb{R}^*_+$ such that $(\r(t_f),\v(t_f)) = (\r_f,\v_f)$ and $m(t_f) \geq M_0$.
%\end{definition}
%
%\begin{definition}[Controllability for DOPs]
%We say that the system $\Sigma_{\mathrm{sat}}$ is controllable for DOPs if for every initial mass $m_i > 0$, every initial point $(\r_i,\v_i)\in\mathcal{P}^{+}$, and every final point $(\r_f,\v_f)\in\mathcal{P}^{-}$, there exists a finite time $t_f \in{\mathcal{I}}_{\Gamma}$, an $M_0> 0$, and an admissible controlled trajectory $(\r(t),\v(t),m(t))=\Gamma(t,\boldsymbol{\tau}(t),\r_i,\v_i,m_i)$ of $\Sigma_{\mathrm{sat}}$ on $[0,t_f]$  in $\mathcal{A}\times\mathbb{R}^*_+$ such that $(\r(t_f),\v(t_f)) = (\r_f,\v_f)$ and $m(t_f) \geq M_0$.
%\end{definition}

\noindent For every initial point $\x_i \in{\mathcal{X}}$ and $\varepsilon>0$, we use $\Gamma_{\varepsilon}(t,\u(t),\x_i)$ to denote the trajectory of $\Sigma_{\varepsilon}$ in Eq.(\ref{EQ:Sigma_varepsilon}) associated with a measurable control $\u(\cdot):[0,\bar{t}_f]\rightarrow \mathcal{B}_{\varepsilon}$ and we define $\bar{t}_f\in\mathbb{R}^+$ as the maximum time such that $\Gamma_{\varepsilon}(t,\u(t),\x_i)$ lies in $\mathcal{X}$ on $[0,\bar{t}_f)$. Set 
$\bar{\mathcal{I}}=[0,\bar{t}_f )$. We refer to $\Gamma_{\varepsilon}(t,\u(t),\x_i)$ as the controlled trajectory of $\Sigma_{\varepsilon}$ starting from $\x_i$ and corresponding to the 
control $\u(\cdot)$.

\begin{remark}
Since $\gamma_{\x}(t)=\Gamma_{\varepsilon}(t,\boldsymbol{0},\x)$ on $\bar{\mathcal{I}}$, the uncontrolled trajectory $\Gamma_{\varepsilon}(t,\boldsymbol{0},\x)$ is periodic on $\mathbb{R}^+$ if $\x\in\mathcal{P}$.
\end{remark}

\begin{lemma}
Fix $\varepsilon > 0$ and $\y_i=(\x_i,m_i)\in{\mathcal{X}}\times\mathbb{R}^*_+$. Then, given every measurable control $\u(\cdot):[0,t_f]\rightarrow\mathcal{B}_{\varepsilon}$, if $\tau_{\mathrm{max}} \geq \varepsilon m_i$, then there exists $M_0 > 0$ and an admissible controlled trajectory $(\x(t),m(t))=\Gamma(t,\boldsymbol{\tau},\x_i,m_i)$ of $\Sigma_{\mathrm{sat}}$ on $[0,t_f]$ in $\mathcal{A}\times[M_0,m_i]$  such that $\Gamma_{\varepsilon}(t,\u,\x_i) = \x(t)$ for every $t\in[0,t_f]$ and $m(t_f) \geq M_0$. 
\label{LE:controllability_mass_varying}
\end{lemma}
\begin{proof}
Since $m(t)\leq m_i$ for each time $t \in[0,t_f]$ and $\tau_{\mathrm{max}} \geq \varepsilon m_i$, it follows that the thrust vector $\boldsymbol{\tau}(\cdot)$ on $[0,t_f]$ can take values in the set $\mathcal{T}$ in Eq.(\ref{EQ:thrust_admissible}) such that $\boldsymbol{\tau}(t)/m(t) = \u(t)$  forevery time $t\in[0,t_f]$. Thus, let $(\x(t),m(t))=\Gamma(t,\boldsymbol{\tau},\x_i,m_i)$, we have $\Gamma_{\varepsilon}(t,\u,\x_i) = \x(t)$ for every time $t\in[0,t_f]$.
Since along the trajectory $(\x(t),m(t))=\Gamma(t,\boldsymbol{\tau},\x_i,m_i)$ on $[0,t_f]$, we have $\u(t)=\boldsymbol{\tau}(t)/m(t)$, which implies that $\dot{m}(t) = -\beta \parallel \u(t)\parallel m(t)$. Thus, we obtain
\begin{eqnarray}
m(t) &=& m_i e^{-\beta \int_{0}^{t_f}\parallel \u(t)\parallel dt}\nonumber\\
 &>& m_ie^{-\beta \varepsilon t_f}.
\end{eqnarray}
Let $M_0: = m_i e^{-\beta \varepsilon t_f} > 0$. Then $m(t_f) \geq M_0$ and the lemma is proved. 
\end{proof}

\noindent In order to study controllability, it is necessary to first show that the admissible region $\mathcal{A}$ is a connected subset of $\mathcal{P}$. 
\begin{lemma}[Connectedness of $\mathcal{A}$]
The admissible region $\mathcal{A}$ is an arc-connected subset of $\mathcal{P}$, i.e., for every initial point $\x_i\in\mathcal{A}$ and every final point $\x_f\in \mathcal{A}$, there exists a continuous path $f:[0,1]\rightarrow \mathcal{A},\ \lambda \mapsto \x(\lambda)$ such that $\x(0)=\x_i$, and $\x(1)=\x_f$.
\label{PR:Connectedness}
\end{lemma}
\begin{proof}
We use the MEOE coordinates (cf. Definition~\ref{meoe}) to prove the result, i.e., it is enough to show that $\mathcal{Z}$ is arc-connected.
Let us choose two point $\z_i$ and $\z_f$ in $\mathcal{Z}$ given by
$$
\z_i=(P_i,e_{x_i},e_{y_i},h_{x_i},h_{y_i},l_i),\ \ \z_f=(P_f,e_{x_f}, e_{y_f},h_{x_f},h_{y_f},l_f),
$$
with $\x_i=\x(\z_i)$ and $\x_f=\x(\z_f)$. We thus define the path $\z:[0,1]\rightarrow\mathcal{Z}$
by $$\z(\lambda)=(P(\lambda),e_x(\lambda) ,e_y(\lambda),h_x(\lambda) ,h_y(\lambda),l(\lambda) )$$ where 
\begin{eqnarray}
P(\lambda) &=& [(1-\lambda) r_i+\lambda r_f ][1 + e_x(\lambda)\cos(l(\lambda)) + e_y(\lambda)\sin(l(\lambda))],\nonumber\\
e_x(\lambda) &=&  (1-\lambda)e_{x_i}+\lambda e_{x_f},\ 
e_y(\lambda) = (1-\lambda)e_{Y_i}+\lambda e_{y_f},\nonumber\\
h_x(\lambda) &=&(1-\lambda)h_{x_i}+\lambda h_{x_f},\ 
h_y(\lambda) =(1-\lambda)h_{h_i}+(1-\lambda)e_{x_i},\nonumber\\
l(\lambda) &=&(1-\lambda)l_i+\lambda l_f,\nonumber
\end{eqnarray}
where $r_f = \parallel \r_f \parallel$ and $r_i = \parallel \r_i \parallel$. Note that $e_x(\lambda)^2 + e_y(\lambda)^2 < 1$ for each $\lambda\in[0,1]$. Let $g(\lambda)=(P(\lambda),e_x(\lambda),e_y(\lambda),h_x(\lambda),h_y(\lambda),l(\lambda))$ on $[0,1]$, we then have that $g(0)=\z_i$ and $g(1)=\z_f$. Consider the continuous function $\x(\lambda)=(\r(\z(\lambda)),\v(\z(\lambda)))$ for $\lambda\in [0,1]$. It follows that $\x(0) = \x_i$ and $\x(1)=\x_f$. It is immediate that $\x(\lambda)\in \mathcal{A}$ for  $\lambda\in [0,1]$. 
Finally, since $P(\lambda ) = (\r(\lambda)\times\v(\lambda))^2/\mu > 0$ and $0\leq e = \sqrt{e_x(\lambda)^2 + e_y(\lambda)^2} < 1$, we have $\r(\lambda)\times\v(\lambda)\neq 0$ and $E(\lambda) = \frac{\v(\lambda)^2}{2} - \frac{\mu}{\parallel \r(\lambda)\parallel} < 0$ on $[0,1]$. This proves the lemma.
\end{proof}

\noindent We also need the following lemma.

\begin{lemma}[Connectedness of $\mathcal{\mathcal{P}^+}$]
The set $\mathcal{P}^+$ is a connected subset of $\mathcal{A}$.
\label{PR:Connectedness_P+}
\end{lemma}
\begin{proof}
Given every two points $\x_i\in\mathcal{P}^+$ and $\x_f\in\mathcal{P}^+$, using the same technique as in the proof of {\it Lemma \ref{PR:Connectedness}}, let $\x(\lambda)=(\r(\z(\lambda)),\v(\z(\lambda)))$ on $[0,1]$, but we rewrite $P(\lambda)$ in the following form,
$$P(\lambda) = ( (1-\lambda) r_{p_i}+\lambda r_{p_f})(1 + \sqrt{e_x(\lambda)^2 + e_y(\lambda)^2}),$$
where $r_{p_i}=r_p(\x_i)$ and $r_{p_f} = r_p(\x_f)$. Then, we have
$$
r_p(\x(\lambda)) = \frac{P(\lambda)}{1 + \sqrt{e_x(\lambda)^2 + e_y(\lambda)^2}} =(1-\lambda) r_{p_i} + \lambda r_{p_f}> r_c.
$$ 
Thus, $\x(\cdot)$ takes values in $\mathcal{P}^+$ and this proves the lemma.
\end{proof}

\subsection{Controllability for OTP}

In this subsection, we first give a controllability property of $\Sigma_{\varepsilon}$ for OTP, then, according to Lemma $\ref{LE:controllability_mass_varying}$, we will establish the controllability of $\Sigma_{\mathrm{sat}}$ for OTP.

\begin{definition}
For every controlled trajectory $\bar{\x}(\cdot)=\Gamma_{\varepsilon}(\cdot,\bar{\boldsymbol{u}},\x_i)$ of $\Sigma_{\varepsilon}$ (where $\varepsilon>0$, $\bar{\boldsymbol{u}}(\cdot):[0,t_f]\rightarrow \mathcal{B}_{\varepsilon}$ is measurable and $\x_i\in\mathcal{X}$), we define 
$$\Sigma^{*}_{\varepsilon}(\bar{\x}):\dot{\boldsymbol{\lambda}}(t) = A(t)\boldsymbol{\lambda}(t) + B(t) {\boldsymbol{u}(t)},$$
the linearised system along $\bar{\x}(\cdot)$ of $\Sigma_{\varepsilon}$ on $[0,t_f]$, where 
$$
A(t)=\f_{\x}(\bar{\x}(t),\bar{\u}(t)),\ \ B(t)=\f_{\u}(\bar{\x}(t),\bar{\u}(t)),
$$
 on $[0,t_f]$.
\end{definition}

\noindent We first have a result of local controllability for the systems $\Sigma_{\varepsilon}$'s around the periodic trajectories of the drift vector field.
\begin{lemma}
Let $\bar{\x} \in  \mathcal{P}^+$. Then, for every $\rho>0$, there exists $\sigma>0$ such that the following properties hold: 
$\mathcal{B}_{\sigma}(\bar{\x})\subset \mathcal{P}^+$ and, for every $\x\in\mathcal{B}_{\sigma}(\bar{\x})$, there exists a controlled trajectory $\Gamma_{\varepsilon}(t,\u(t),\bar{\x})$ of $\Sigma_{\varepsilon}$  such that
$$\Gamma_{\varepsilon}(0,\u,\bar{\x})=\bar{\x},\ \ \Gamma_{\varepsilon}(t_p(\bar{\x}),\u,\bar{\x})={\x},$$
and
$$\parallel \Gamma_{\varepsilon}(t,\u,\bar{\x}) - \Gamma_{\varepsilon}(t,\boldsymbol{0},\bar{\x})\parallel < \rho,$$
for $t\in[0,t_p(\bar{\x})]$.
\label{LE:Local_Controllability_Along_Periodic_Trajectory}
\end{lemma}
\begin{proof}
\noindent According to Theorem 7 of Chapter 3 in Ref.\ \cite{Sontag:98}, it suffices to prove the controllability of the linearized system $\Sigma^{*}_{\varepsilon}(\Gamma_{\varepsilon}(t,\boldsymbol{0},\bar{\x}))$ along the periodic trajectory $\Gamma_{\varepsilon}(t,\boldsymbol{0},\bar{\x})$ on the interval $[0,t_p(\bar{\x})]$. Then, the latter controllability would follow, according to Corollary 3.5.18 of Chapter 3 in Ref.\ \cite{Sontag:98}, by the following rank condition: there exists a time $\tau\in[0,t_p(\bar{\x})]$ and a nonnegative integer $k$ such that the rank of the matrix $[B_0(\tau),B_1(\tau),\cdots,B_k(\tau)]$ equals $6$, where $B_{i+1}(t) = A(t)B_i(t) - \frac{d}{dt}B_i(t)$ for $i=1,2,\cdots$, and $B_0(t) = B(t)$.
It therefore amounts to compute some $B_i(\cdot)$'s. The explicit expressions for matrices $A$ and $B$ in terms of $\x$ are
$$A=\f_{\x} = \left[\begin{array}{cc}\boldsymbol{0}&I_3\\
-\frac{\mu}{\parallel \r \parallel^3}I_3 + 3 \frac{\mu}{\parallel \r \parallel^5}\r\cdot\r^T&\boldsymbol{0}\end{array}\right],\ \text{and}\ B=\f_{\u}=\left[\begin{array}{c}\boldsymbol{0}\\
I_3\end{array}\right].$$
Since $B_0(t)=B(t)$, it follows that
\begin{eqnarray}
B_1(t) &=& A(t)B_0(t) - \frac{d}{dt}B_0(t) \nonumber\\
&=& A(t)B_0(t) = \left[\begin{array}{cc}I_3&
\boldsymbol{0} \end{array}\right]^T\nonumber.
\end{eqnarray}
Thus, we have that the rank of the matrix $[B_0(t),B_1(t)] = \left[\begin{array}{cc} \boldsymbol{0} & I_3\\ I_3 & \boldsymbol{0} \end{array}\right]$ is equal to $6$ for every time $t\in[0,t_p(\bar{\x})]$, proving the lemma.
\end{proof}

\begin{proposition}
For $\varepsilon >0$, the control system $\Sigma_{\varepsilon}$ is controllable for OTP within $\mathcal{P}^+$, i.e., for every initial point $\x_i\in\mathcal{P}^+$ and final point $\x_f\in\mathcal{P}^+$, there exists a controlled trajectory $\Gamma_{\varepsilon}(t,\u,\x_i)$ of $\Sigma_{\varepsilon}$  in $\mathcal{P}^+$ on a finite interval $[0,t_f]\subset\bar{\mathcal{I}}$ such that $\Gamma_{\varepsilon}(t_f,\u,\x_i)=\x_f$.
\label{PR:Controllability_constant_mass}
\end{proposition}
\begin{proof}
Since the subset $\mathcal{P}^+$ is path-connected as is shown by {\it Proposition} \ref{PR:Connectedness_P+}, it follows that any two different points $\x_i$ and $\x_f$ in $\mathcal{P}^+$ are connected by a path $\x:[0,1]\rightarrow \mathcal{P}^+,\ \lambda\mapsto \x(\lambda)$ such that $\x(0) = \x_i$ and $\x(1) = \x_f$. By compactness of the support of $\x(\cdot)$ in $\mathcal{P}^+$ there exists, for every $\sigma>0$, a finite sequence of points $\x_0,\x_1,\cdots,\x_N$, on the support of $\x(\lambda)$ so that  $\x_0 = \x_i$, $\x_n=\x_f$ and $\x_{j+1}\in\mathcal{B}_{\sigma}(\x_j)$, for $j=0,1,\cdots,N-1$.
According to $Lemma$ $\ref{LE:Local_Controllability_Along_Periodic_Trajectory}$,  for every $\rho>0$, there exists $\sigma >0$ small enough and a finite sequence of points  $\x_0,\x_1,\cdots,\x_N$ as above such that, for $j=0,1,\cdots,N-1$, 
$\x_{j+1}\in\mathcal{B}_{\sigma}(\x_j)$ and one has a controlled trajectory $\Gamma_{\varepsilon}(t,\u_j,\x_j)$ on the interval $[0,t_p(\x_j)]$ such that
$$
\Gamma_{\varepsilon}(0,\u_j,\x_j) = \x_j,\ \ \Gamma_{\varepsilon}(t_p(\x_j),\u_j,\x_j) = \x_{j+1},
$$
and
$$
\parallel \Gamma_{\varepsilon}(t,\u_j(t),\x_j) - \Gamma_{\varepsilon}(t,\boldsymbol{0},\x_j)\parallel < \rho, \hbox{ for }t\in [0,t_p(\x_j)].
$$
For $\sigma >0$, let $\mathcal{W}_j\subset\mathcal{P}^+$ be an open neighborhood of $\x_j$ such that $\mathcal{B}_{\sigma}(\x_j)\subset \mathcal{W}_j$ for $j=0,1,\cdots,N$ and set $\mathcal{P}^+_{\mathcal{W}_j} = \{\Gamma_{\varepsilon}(t,0,{\mathcal{W}_j}),\ t\geq 0\}$. For $\rho>0$ small enough, the open set $\mathcal{P}^+_{\mathcal{W}_j}$ is included in $\mathcal{P}^+$. By concatenating the $\Gamma_{\varepsilon}(\cdot,\u_j,\x_j)$ for $j=0,1,\cdots,N-1$, the  initial point $\x_i$ can be steered to $\x_f$, proving the proposition.
\end{proof}

\noindent According to Lemma~\ref{LE:controllability_mass_varying}, and recalling the definition of controllability for OTPs in Definition~\ref{DE:OTP}, we obtain the following result of controllability:
\begin{corollary}
For every $\mu>0$, $\beta > 0$, $\tau_{\mathrm{max}} > 0$, the system $\Sigma_{\mathrm{sat}}$ is controllable for OTP.
\label{PR:Controllability_varying_mass}
\end{corollary}

\noindent Note that $\mathcal{P}^+\subset\mathcal{A}$, so the system $\Sigma_{\mathrm{sat}}$ is controllable for OTPs within $\mathcal{A}$ no matter what value of $\tau_{\mathrm{max}}$ the low-thrust control system can provide if the satellite takes high enough percent of total fuel, i.e., $(m_i - m_f)/m_i > 0$ is big enough, which makes senses in engineering for electric thrust systems whose maximum thrust $\tau_{\mathrm{max}}$ is very small.

\subsection{Controllability for OIP}\label{SE:P-_P+}

%Assume a satellite, with positive initial mass $m_i$, is placed on a point $\x_i$ in $\mathcal{P}^-$ at initial time $t=0$. Since $r_p(\x_i) < r_c$, it is clear that if $T_{\mathrm{max}}$ is not big enough then, for every control $\boldsymbol{\tau}(\cdot):[0,t_p(\x_i)]\rightarrow\mathcal{T}$ and corresponding controlled trajectory $(\x(t),m(t))=\Gamma(t,\boldsymbol{\tau},\x_i,m_i)$ on $[0,t_p(\x_i)]$, there exists an interval $[{t}_1,{t}_2]\subset[0,t_p(\x_i)]$ such that $\parallel \r(t)\parallel < r_c$ on $[{t}_1,{t}_2]$. Hence, the system $\Sigma_{\mathrm{sat}}$ is not controllable for OIPs for every $T_{\mathrm{max}}$.

We provide next a controllability criterium for OIP.
\begin{lemma}
Assume that, for every point $(\x_i,m_i)\in\mathcal{P}^-\times\mathbb{R}^*_+$, there exists $\tau>0$ and a positive time $\bar{t}\in{\mathcal{I}}_{\Gamma}$ and a control $\tilde{\boldsymbol{\tau}}(\cdot):[0,\bar{t}]\rightarrow \mathcal{T}(\tau)$ such that along the controlled trajectory $(\tilde{\x},\tilde{m}(t))= \Gamma(t,\tilde{\boldsymbol{\tau}}(t),\x_i,m_i)$ on $[0,\bar{t}]$,  we have $\tilde{\x}(t)\in\mathcal{A}$ on $[0,\bar{t}]$, $\tilde{m}(\bar{t})>0$, and  $r_{p}(\tilde{\x}(\bar{t})) > r_c$. Then, the system $\Sigma_{\mathrm{sat}}$ is controllable for OIP from 
$(\x_i,m_i)$. 
\label{LE:OIP1}
\end{lemma}
\begin{proof}
Note that the assumption implies that there exists a control $\tilde{\boldsymbol{\tau}}(\cdot):[0,\bar{t}]\rightarrow\mathcal{T}(\tau)$ such that the admissible controlled trajectory $(\tilde{\x}(t),\tilde{m}(t))= \Gamma(t,\tilde{\boldsymbol{\tau}},\x_i,m_i)$ in $\mathcal{A}\times[m(\bar{t},m_i]$ on $[0,\bar{t}]$ steers $(\x_i,m_i)$ in $\mathcal{P}^-\times\mathbb{R}^*_+$ to some $(\x(\bar{t}),m(\bar{t}))$ in $\mathcal{P}^+\times\mathbb{R}^*_+$. After arriving at $\x(\bar{t})$ in $\mathcal{P}^+$, according to Proposition~\ref{PR:Controllability_constant_mass}, it follows that there exists an $M_0\in(0,m(\bar{t})]$, a finite time $t_f\in{\mathcal{I}}_{\Gamma}$, and a control $\boldsymbol{\tau}(\cdot):[0,t_f]\rightarrow\mathcal{T}(\tau)$  such that along the controlled trajectory $(\x(t),m(t))=\Gamma(t,\boldsymbol{\tau},\tilde{\x}(\tau),\tilde{m}(\tau))$ on $[0,t_f]$, we have $\x(t)\in\mathcal{A}$ on $[0,t_f]$, $\x(t_f)=\x_f$, and $m(t_f) > M_0$.
\end{proof} 

\noindent One cannot have controllability for OIP for every value of $\tau_{\mathrm{max}}>0$. Indeed, pick a point $(\x_i,m_i)$ in $\mathcal{P}^-\times\mathbb{R}^*_+$. For every control $\boldsymbol{\tau}(\cdot)$
taking values in $\mathcal{T}(\tau_{\mathrm{max}})$, the corresponding controlled trajectory $(\x(\cdot),m(\cdot))=\Gamma(\cdot,\boldsymbol{\tau},\x_i,m_i)$ converges to $\Gamma(\cdot,\boldsymbol{0},\x_i,m_i)$  on $[0,t_p(\x_i)]$ as $\tau_{\mathrm{max}}$ tends to zero. Then, since $r_p(\x_i) < r_c$, there exists $t\in[0,t_p(\x_i)]$ such that $\parallel \r(t)\parallel < r_c$ implying that $(\x(\cdot),m(\cdot))=\Gamma(\cdot,\boldsymbol{\tau},\x_i,m_i)$ is not admissible for every control $\boldsymbol{\tau}(\cdot)$ taking values in $\mathcal{T}(\tau_{\mathrm{max}})$. For large values of $\tau_{\mathrm{max}}$, we can steer $(\x_i,m_i)\in\mathcal{P}^-\times\mathbb{R}^*_+$ to $(\x_f,m_f)\in\mathcal{P}^+\times\mathbb{R}^*_+$ as described in the following lemma.

%Let us define the following time-optimal control problem. We want to minimize the transfer time, i.e., $t\rightarrow \text{min}$, when we control a satellite subject to the system $\Sigma_{\mathrm{sat}}$ from the initial point $\y_i=(\x_i,m_i)$ in $\mathcal{P}^-\times\mathbb{R}^*_+$ to a final submanifold $$\mathcal{M}=\big\{(\x_f,m_f)\in\mathcal{P}\times\mathbb{R}^*_+:\parallel \r_f \parallel - r_p(\x_f) = 0, r_p(\x_f) - r_c = 0\big\}.$$
%We define by $t_{min}$ the solution of the time-optimal problem, it is obvious that $t_{min}$ is a function of $T_{\mathrm{max}} > 0$ such that $\frac{d t_{min}}{d T_{\mathrm{max}}} < 0$ if we fixing the initial point $\y_i$.
\begin{lemma}
For every $\beta,\mu>0$ and point $\y_i=(\x_i,m_i)$ in $\mathcal{P}^-\times\mathbb{R}^*_+$, there exists  $\tau_{\mathrm{max}}$ such that the following holds:
%, there exists a controlled trajectory $(\x(\cdot),m(\cdot))=\Gamma(\cdot,\boldsymbol{\tau},\y_i,m_i)$ of $\Sigma_{\mathrm{sat}}$ defined on an interval of time $[0,T]$ for some $T>0$, such that, for every $t\in [0,T]$, one has $\x(t)\in\mathcal{A}$, $\Vert \boldsymbol{\tau}(t)\Vert \leq T_{\mathrm{max}}$ and $r_p(\x(T))>r_c$.
%$\tilde{T}_{\mathrm{max}}> 0$ depending on $\y_i$ such that the following properties hold:
\begin{description}
\item[$1)$] if $\tau> \tau_{\mathrm{max}}$, there exists a control $\boldsymbol{\tau}(\cdot)$ taking values in  $\mathcal{T}(\tau)$ and a positive time $\bar{t}\in{\mathcal{I}}_{\Gamma}$ such that, along the controlled trajectory $(\x(t),m(t))=\Gamma(t,\boldsymbol{\tau},\y_i,m_i)$ on $[0,\bar{t}]$, we have $\x(t)\in\mathcal{A}$ on $[0,\bar{t}]$, $m(\bar{t}) > 0$, and $r_p(\x(\bar{t})) > r_c$;
\item[$2)$] if $\tau\leq \tau_{\mathrm{max}}$, for every control $\boldsymbol{\tau}(\cdot)$ taking values in  $\mathcal{T}(\tau)$, the controlled trajectory $(\x(\cdot),m(\cdot))=\Gamma(\cdot,\boldsymbol{\tau},\y_i,m_i)$ does not reach $\mathcal{P}^+\times\mathbb{R}^*_+$.
\end{description}
\label{PR:OIP1}
\end{lemma}
\begin{proof}{At the light of the remark preceding the lemma, it is enough to find a value of $\bar{\tau}$ and a control $\boldsymbol{\tau}(\cdot)$ taking values in  $\mathcal{T}(\bar{\tau})$
steering $\y_i=(\x_i,m_i)$ to some point in $\mathcal{P}^+\times\mathbb{R}^*_+$ along an admissible trajectory. 
We proceed as follows. Let $C_0=\Vert \r_i\Vert^{1/2} \Vert \v_i\Vert$ which belongs to $(0,{\sqrt{2\mu}})$. 
Choose now the function $\v(\cdot)$ defined on some time interval $[0,\bar{T}]$ (with $\bar{T}$
to be fixed later) as follows: $\v(0)=\v_i$, $\dot \r(t)=\v(t)$ and
$$
\v(t)=\frac{C_0}{(2\r(t))^{1/2}}\left(a_i\frac{\r(t)}{\Vert\r(t)\Vert}+b_i\frac{\r(t)^{\perp}}{\Vert\r(t)\Vert}\right), \ t\in [0,T].
$$
Here $\r(t)^{\perp}$ denotes a continuous choice of vector perpendicular to $\r(t)$ in the 2D plane spanned by $\r_i$ and $\v_i$. (Implicitely, we assume with no loss of generality that $\v_i\neq 0$ with $a_i>0$ and $b_i\neq 0$.) 
For simplicity, we assume next that $a_i=b_i=1$. The curve $(\r(\cdot),\v(\cdot))$ defined previously can be explicitely integrated using polar coordinates for $\r(\cdot)=r(\cdot)\exp(i\theta(\cdot))$. One gets that, on $[0,T]$, 
$$
\dot r(t)=r(t)\dot\theta(t)=\left(\frac{C_0}{2r(t)}\right)^{1/2}.
$$
After integration, one has for $t\in [0,T]$, 
$$
r(t)=(r_i^{3/2}+3C_0t/2)^{2/3}, \ \ \theta(t)=\theta_i+2\ln(r_i^{3/2}+3C_0t/2)/3.
$$
One also checks that $\h(t)=C_0(\frac{r(t)}2)^{1/2}\e_i$, with $\e_i$ a constant vector of unit norm parallel to $\r_i\times \v_i$ and ${\boldsymbol{L}}(t)=(C_0^2/2-\mu)\frac{\r(t)}{\Vert\r(t)\Vert}$. One deduces that 
$r_p(t)=C_1r(t)^{1/2}$ for some positive constant $C_1$. One thus fixes $\bar{T}$ so that 
$r_p(\bar{T})>r_c$. 
It remains to determine $\bar{\tau}$ so that $(\r(t),\v(t))$ is part of a controlled admissible trajectory of System $\Sigma_{\varepsilon}$. 
One first integrates over $[0,T]$ the differential equation 
$$
\dot m(t)=-\beta\Vert \dot {\v}(t)+\frac{\mu}{r(t)^3}\r(t)\Vert m(t),
$$
and then take $\boldsymbol{\tau}(t)=\left(\dot {\v}(t)+\frac{\mu}{r(t)^3}\r(t)\right)m(t)$ for $t\in [0,T]$. The final bound $\bar{\tau}$ is simply the maximum of $\Vert \boldsymbol{\tau}(t)\Vert$ over $[0,T]$.

}
\end{proof}
%Note that the equation $\tau(T_{\mathrm{max}}) - \tau_r(T_{\mathrm{max}}) = 0$ has only one zero for $T_{\mathrm{max}}\in(0,+\infty)$, a simple bisection method can quickly compute it root $\tilde{T}_{\mathrm{max}}$.

\noindent As a combination of Lemmas~\ref{LE:OIP1} and ~\ref{PR:OIP1}, we obtain the following result.
\begin{corollary}
For every $\beta > 0$, $\mu > 0$, and initial point $(\x_i,m_i)\in\mathcal{P}^-\times\mathbb{R}^*_+$, there exists a limiting value $\tau_{\mathrm{max}} > 0$ depending on $(\x_i,m_i)$ such that the following properties hold:
\begin{description}
\item[$1)$] if $\tau > \tau_{\mathrm{max}}$, the system $\Sigma_{\mathrm{sat}}$ is controllable for the OIP; and
\item[$2)$] if $\tau \leq \tau_{\mathrm{max}}$, the system $\Sigma_{\mathrm{sat}}$ is not controllable for the OIP.
\end{description}
\label{PR:specific1}
\end{corollary}
\noindent The limiting value $\tau_{\mathrm{max}}$ can be computed by combining a shooting method and a bisection method as described in Section~\ref{SE:Numerical}.

\subsection{Controllability for DOP}\label{SE:P+_P-}

Let us define a system $\tilde{\Sigma}_{\mathrm{sat}}$ associated to $\Sigma_{\mathrm{sat}}$ as
\begin{eqnarray}
\tilde{\Sigma}_{\mathrm{sat}}:
\begin{cases}
\dot{\r}(t) = \v(t),\\
\dot{\v}(t) = -\frac{\mu}{\parallel \r(t)\parallel^3}\r(t) + \frac{\boldsymbol{\tau}(t)}{m(t)},\\
\dot{m}(t) = + \beta \parallel \boldsymbol{\tau}(t)\parallel,
\end{cases}
\label{EQ:tilde_Sigma}
\end{eqnarray}
where all variables are the same as defined in Eq.~(\ref{EQ:Sigma}). For every $\y_i\in\tilde{\mathcal{X}}\times\mathbb{R}^*_+$ and measurable control function $\boldsymbol{\tau}$ taking values in $\mathcal{T}(\tau_{\mathrm{max}})$ with $\tau_{\mathrm{max}}$,  we define by $\tilde{\Gamma}(t,\boldsymbol{\tau}(t),\y_i)$ the corresponding trajectory of $\tilde{\Sigma}_{\mathrm{sat}}$ for some positive times. 

\begin{remark}
For every controlled trajectory $(\r(t),\v(t),m(t))=\Gamma(t,\boldsymbol{\tau}(t),\y_i)$ of the system $\Sigma_{\mathrm{sat}}$ on some finite intervals $[0,t_f]\subset{\mathcal{I}_{\Gamma}}$ with $
(\r_f,\v_f,m_f)=\Gamma(t_f,\boldsymbol{\tau}(t_f),\y_i),$
the trajectory  $(\tilde{\r}(t),\tilde{\v}(t),\tilde{m}(t))=\tilde{\Gamma}(t,\boldsymbol{\tau}(t_f -t),\r_f,-\v_f,m_f)$ of $\tilde{\Sigma}_{\mathrm{sat}}$ runs backward in time along the trajectory $\Gamma(t,\boldsymbol{\tau}(t),\y_i)$ on $[0,t_f]$, i.e.,
\begin{eqnarray}
(\tilde{\r}(t),\tilde{\v}(t),\tilde{m}(t)) = (\r(t_f - t),-\v(t_f-t),m(t_f-t))\ \text{on}\ [0,t_f].
\label{EQ:tilde_Sigma=Sigma}
\end{eqnarray}
\end{remark}
%
%
%\begin{proof}
%Since at initial time $t=0$, we have $\tilde{\r}(0) = \r(t_f)$ and $\tilde{\v}(0) = - \v(t_f)$, it suffices to prove that the time derivatives of both sides of Eq.(\ref{EQ:forward_backward}) at each time $t\in[0,t_f]$ equal each other. Taking the time derivatives of $\tilde{\r}(t)$ and $\r(t_f - t)$ on $[0,t_f]$, we have
%$$\frac{d\tilde{\r}(t)}{dt} = \tilde{\v}(t),\ \frac{d \r(t_f - t)}{dt} = -\v(t_f - t).$$
%Since $\tilde{\v}(t) = -\v(t_f - t)$, we have 
%$$\frac{d\tilde{\r}(t)}{dt} = \frac{d \r(t_f - t)}{dt}.$$
%Taking the time derivatives of $\tilde{\v}(t)$ and $\v(t_f - t)$ on $[0,t_f]$, we have
%\begin{eqnarray}
%\frac{d \tilde{\v}(t)}{dt} &=& -\frac{\mu}{\parallel \tilde{\r}(t)\parallel^3 }\tilde{\r}(t) + \u(t_f - t),\nonumber\\
% \frac{d \v(t_f - t)}{dt} &=& -\dot{\v}(t_f - t)\nonumber\\
% &=& -\left(-\frac{\mu}{\parallel \r(t_f - t)\parallel^3}\r(t_f - t) + \u(t_f - t)\right).\nonumber
%\end{eqnarray}
%Since $\tilde{\r}(t) = \r(t_f - t)$, we have
%$$ \frac{d\tilde{\v}(t)}{dt}  = - \frac{d \v(t_f - t)}{dt}.$$
%Thus, the lemma is proved.
%\end{proof}

\noindent As a consequence, according to Lemma~\ref{PR:OIP1} and Corollary~\ref{PR:specific1}, we obtain the following result.
\begin{corollary}
For each $\beta > 0$, $\mu > 0$, and $m_i > 0$, given a point $(\r_f,\v_f)\in\mathcal{P}^-$, there exists a $\tau_{\mathrm{max}} > 0$ depending on $(\r_f,\v_f)$ and $m_i$ such that the following properties hold:
\begin{description}
\item[$1)$] if $\tau > \tau_{\mathrm{max}}$, the system $\Sigma_{\mathrm{sat}}$ is controllable for the corresponding DOP; and

\item[$2)$]  if $\tau\leq\tau_{\mathrm{max}}$, the system $\Sigma_{\mathrm{sat}}$ is not controllable for the corresponding DOP.
\end{description}
\label{CO:DOP}
\end{corollary}

\section{{Numerical Examples}}\label{SE:Numerical}

In this section, we consider two numerical examples, one OIP and one DOP, to  compute the limiting value $\tau_{\mathrm{max}}$ in Corollaries \ref{PR:specific1} and \ref{CO:DOP}, respectively. The gravitational constant $\mu$ in system $\Sigma_{\mathrm{sat}}$ (and/or $\tilde{\Sigma}_{\mathrm{sat}}$) is $3986000.47$ Km$^3/$s$^2$, the radius of the Earth is $r_e = 6,374,000$ m, and we consider the vertical depth of atmosphere around the Earth is $90,000$ m, which means $r_c = r_e + 90,000$ m.

\subsection{A numerical example for OIP}

In order to be able to compute the limiting value $\tau_{\mathrm{max}}$ in Corollary \ref{PR:specific1}, we first define the following optimal control problem.
\begin{definition}[{\it Optimal control problem (OCP) for OIP} ]
Given every initial point $(\x_i,m_i)\in\mathcal{P}^-\times\mathbb{R}^*_+$ and $\tau>0$, the optimal control problem for OIP consists of  to steering a satellite by $\boldsymbol{\tau}(\cdot)\in\mathcal{T}(\tau)$ on a time interval $[0,t_f]\subset\mathcal{I}_{\Gamma}$ along System $\Sigma_{\mathrm{sat}}$ such that, along the controlled trajectory $(\r(t),\v(t),m(t))=\Gamma(t,\boldsymbol{\tau}(t),\x_i,m_i)$, the time $t_f$ is the first occurence for $\parallel\r(t_f) \parallel = r_p(\r(t_f),\v(t_f))$, i.e., $\parallel\r(t) \parallel > r_p(\r(t),\v(t))$ on $[0,t_f)$, and that $r_p(\r(t_f),\v(t_f))$  is maximized, i.e., the cost functional is
\begin{eqnarray}
J = \int_0^{t_f} \frac{d}{dt}r_p(\r(t),\v(t))dt.
\label{EQ:controllability_cost}
\end{eqnarray}
\end{definition}
\noindent Let $\tilde{t}_f > 0$ be the optimal final time of the OCP for OIP, and let $(\tilde{\x}(t),\tilde{m}(t)) = \Gamma(t,\tilde{\boldsymbol{\tau}}(t),\x_i,m_i)$ on $[0,\tilde{t}_f]$ be the optimal controlled trajectory with the associated optimal control $\tilde{\boldsymbol{\tau}}(t)\in\mathcal{T}(\tau)$ on $[0,\tilde{t}_f]$. One can check, by using Pontryagin Maximum Principle as was done in Ref.\ \cite{Caillau:01}, that $\parallel \tilde{\boldsymbol{\tau}}(t)\parallel = \tau$ on the whole interval $[0,\tilde{t}_f]$. Thus, fixing the initial point $(\x_i,m_i)\in\mathcal{P}^-\times\mathbb{R}^*_+$,  we have that the final time $\tilde{t}_f$, the trajectory $(\tilde{\x}(t),\tilde{m}(t))$ at each time $t\in[0,\tilde{t}_f]$, and the final perigee distance $r_p(\tilde{\x}(\tilde{t}_f))$ are functions of $\tau$. Thus, let us define a function 
\begin{eqnarray}
s:\mathbb{R}_+ \rightarrow \mathbb{R}, s(\tau) = r_p(\tilde{\x}(\tilde{t}_f)) - r_c.
\label{EQ:shooting_OIP}
\end{eqnarray}
If one can find ${\tau}_{\mathrm{max}} > 0$ such that $s({\tau}_{\mathrm{max}}) =0$, then $\tau_{\mathrm{max}}$ is the limiting value in Corollary \ref{PR:specific1}. For every $\tau> 0$, using a shooting method as an inner loop to solve the OCP for OIP, we can obtain a value for $s(\tau)$. Then, using a bisection method as an outer loop, one can obtain ${\tau}_{\mathrm{max}}>0$ such that $s({\tau}_{\mathrm{max}}) = 0$.
According to Eq.(\ref{EQ:r_max}) and the objective of the OCP for OIP, place a satellite with the initial mass $m_i>0$ on a point $(\r_i,\v_i)\in\mathcal{P}^-$. The optimal controlled trajectory lies on a 2-dimensional plane spanned by $\r_i$ and $\v_i$. Hence, the limiting value $\tau_{\mathrm{max}}$ in Corollary \ref{PR:specific1} is determined only by $\parallel \r_i \parallel$, $\parallel \v_i \parallel$, and by the flight path angle $\eta_i \in [-\pi/2,\pi/2]$,  i.e., the angle between the velocity vector $\v_i$ and local horizontal plane, defined by $$\eta _i= \sin^{-1}\left(\frac{\r_i^T \cdot \v_i}{\parallel \r_i \parallel \parallel \v_i \parallel}\right).$$
Assume that a rocket carries a satellite, whose initial mass is $m_i =150$ $kg$, from the surface of the Earth to a point $\x_i=(\r_i,\v_i)$ in the unstable region $\mathcal{P}^-$ such that $\parallel \r_i \parallel = r_e + 110,000$ $m$, $\parallel \v_i \parallel =  7879.5$ $m/s$, and  $\eta_i =5^{\circ}$. The rocket and the satellite are separated at this point $\x_i$. Then, the satellite has to use its own engine to steer itself from the point $\x_i$ into the stable region $\mathcal{P}^+$.  We can see from Figure \ref{Fig:trajectory_inserting} that the periodic orbit $\gamma_{\x_i}$ has collisions with the surface of the atmosphere around the Earth.
\begin{figure}[!ht]
\centering
\includegraphics[trim=3.2cm 7.2cm 3.2cm 7.2cm, clip=true, width=3.0in, angle=0]{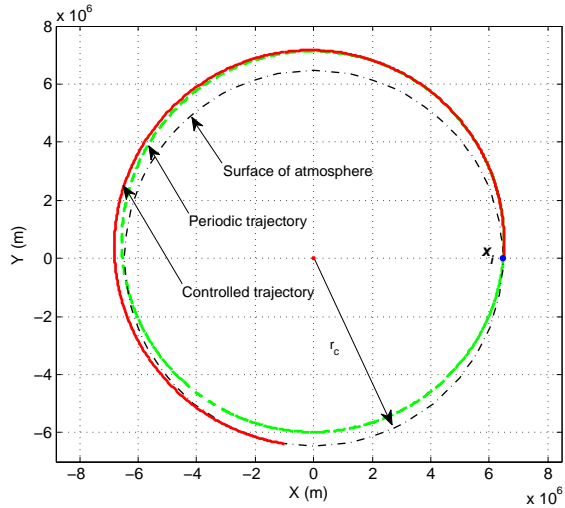}
\caption{The periodic orbital $\gamma_{\x_i}$ and the optimal controlled trajectory of the OCP for the OIP with $\tau = \tau_{\mathrm{max}}$ starting from $\x_i$.}
\label{Fig:trajectory_inserting}
\end{figure}

\begin{figure}[!ht]
\centering
\includegraphics[trim=3.4cm 8.3cm 3.3cm 8.3cm, clip=true, width=3.0in, angle=0]{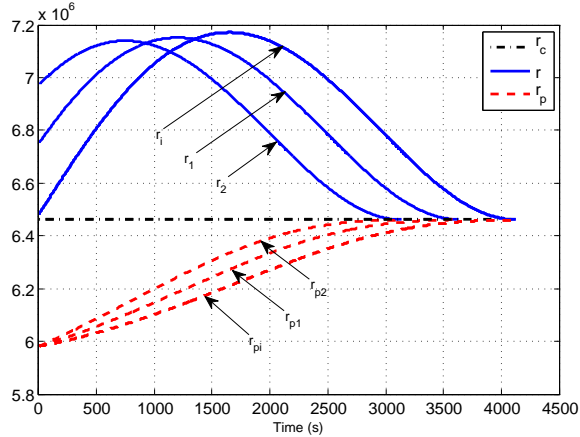}
\caption{The profile of $\parallel \r \parallel $ and $r_{p}$ with respect to time for 3 different initial points, i.e., $\x_i$ and $\x_j$ ($j= 1,2$), for OIP; The subscripts, $i$, $1$, and $2$, correspond to the initial points $\x_i$, $\x_1$, and $\x_2$, respectively.}
\label{Fig:r_rm_inserting}
\end{figure}

We choose the specific impulse of the engine fixed on the satellite as $I_\mathrm{sp}
= 2000$ s, which implies that $\beta = \frac{1}{I_\mathrm{sp}g_0} = 5.102\times 10^{-5}$ m$^{-2}$ where $g_0 = 9.8$ m$^2/$s.  The computed result of the limiting value is $\tau_{\mathrm{max}} = 8.052$ $N$. Thus, in order to be able to insert the satellite from the point $\x_i$ into stable region $\mathcal{P}^+$, the maximum thrust  of the engine has to be larger than $8.052$ $N$. The optimal controlled trajectory of the corresponding OCP for OIP with $\tau = 8.052$  is  plotted in Figure \ref{Fig:trajectory_inserting} as well.
To see the numerical results for different initial points, another two points $\x_1=(\r_1,\v_1)$ and $\x_2=(\r_2,\v_2)$ are chosen on the periodic orbit $\gamma_{\x_i}$ such that 
$$\parallel \r_1 \parallel = r_e + 379,494\ m,\ \parallel \v_1 \parallel = 7,562\ m/s,\ \eta_1 = 4.3517^\circ,$$
$$\parallel \r_2 \parallel = r_e+599,351\ m,\ \parallel \v_2 \parallel = 7,312\ m/s,\ \eta_2 = 3.0132^\circ.$$
Then, the limiting value of $\tau_{\mathrm{max}}$ corresponding to the two initial points $\x_j$, $j=1,2$, are computed as $9.037$ $N$ and $10.719$ $N$, respectively. We see that the limiting values $\tau_{\mathrm{max}}$ are different for different initial points on the same periodic orbit. The time history of radius $\parallel \r(t) \parallel$ and perigee distance $r_{p}(\x(t))$ along the optimal controlled trajectories starting from the initial points $\x_i$ and $\x_j$, $j=1,2$, are plotted in Figure \ref{Fig:r_rm_inserting}. Since the three points $\x_i$ and $\x_j$, $j=1,2$ lie on the same periodic orbit $\gamma_{\x_i}$, $r_p$ at initial time is the same, as shown in Figure  \ref{Fig:r_rm_inserting}.

\subsection{A numerical example for DOP}

A DOP is a powered flight phase of a satellite in the region $\mathcal{A}\times\mathbb{R}^*_+$, during which a decelerating manoeuvre is performed so that the satellite will move to the desired final point $\x_f = (\r_f,\v_f)\in\mathcal{P}^-$ at the entry interface (EI). The condition at EI permits the satellite to have a subsequent safe entry flight in atmosphere to a landing site. A typical condition at EI, see Ref.\ \cite{Baldwin:12}, is given as:
\begin{eqnarray}
\parallel \r_f \parallel = r_{EI},\ \ \parallel \v_f \parallel = V_{EI},\ \ \text{and}\ \ \r_f^T \cdot \v_f = V_{EI} r_{EI} sin(\eta_{EI}),
\label{EQ:DOP_EI_Condition}
\end{eqnarray}
where $r_{EI} = r_e + 122,000$ m, $V_{EI}=7879.5\ m/s$, and $\eta_{EI} = -15^\circ$ denote the norm of position vector, the norm of velocity vector, and the flight path angle at EI, respectively.
In order to compute the limiting value $\tau_{\mathrm{max}}$ in {\it Corollary \ref{CO:DOP}} for the DOP to a point $(\r_f,\v_f)$ in $\mathcal{P}^-$, we first define the below optimal control problem.
\begin{definition}[{\it Optimal control problem (OCP) for DOP} ]
Given every final point $\x_f=(\r_f,\v_f)\in\mathcal{P}^-$ and $\tau > 0$, let $m_i>0$ be the initial mass of a satellite, the optimal control problem for DOP consists of steering the satellite by $\boldsymbol{\tau}(\cdot)\in\mathcal{T}(\tau)$ on a time interval $[0,t_f]\subset\mathcal{I}_{\Gamma}$ subject to the system $\tilde{\Sigma}_{\mathrm{sat}}$ such that, along the controlled trajectory $(\r(t),\v(t),m(t)) = \tilde{\Gamma}(t,\boldsymbol{\tau}(t),\r_f,-\v_f,m_f)$ ($m_f > 0$ is free) of System $\tilde{\Sigma}_{\mathrm{sat}}$, the time $t_f$ is the first occurence for $\parallel\r(t_f) \parallel = r_p(\r(t_f),\v(t_f))$, $m_i = m(t_f)$, and $r_p(\r(t_f),\v(t_f))$  is maximized, i.e., the cost functional is the same as Eq.(\ref{EQ:controllability_cost}).
\end{definition}

\noindent Given every initial mass $m_i>0$ and final point $(\r_f,\v_f)$ in $\mathcal{P}^-$, let $\bar{t}_f>0$ be the optimal final time of the OCP for DOP, and let $(\bar{\x}(t),\bar{m}(t)) = \tilde{\Gamma}(t,\bar{\boldsymbol{\tau}}(t),\r_f,-\v_f,m_f)$ be the optimal controlled trajectory associated to the control $\bar{\boldsymbol{\tau}}(t)\in\mathcal{T}(\tau)$ on $[0,\bar{t}_f]$. Then, the same as the OCP for OIP, the perigee distance $r_p(\bar{\x}(\bar{t}_f))$ is a function of $\tau$. Let us define a function
\begin{eqnarray}
\bar{s}:\mathbb{R}_+ \rightarrow \mathbb{R},\ \bar{s}(\tau) = r_p(\bar{\x}(\bar{t}_f)) - r_c.
\label{EQ:shooting_DOP}
\end{eqnarray}
Then, according to Eq.~(\ref{EQ:tilde_Sigma=Sigma}), in order to compute the limiting value $\tau_{\mathrm{max}}$ in {\it Corollary \ref{CO:DOP}}, it suffices to combine a shooting method and a bisection method to compute the value $\tau_{\mathrm{max}}$ such that $\bar{s}(\tau_{\mathrm{max}}) = 0$.

What we developed in this paper is applicable not only for low-thrust control systems but also for high-thrust control systems if only the thrust is finite instead of impulsive. Thus, we consider the space shuttle's parameters in Refs.\ \cite{Joosten:85,Brand:73}. The initial mass is $95,254.38$ $kg$. The specific impulse of the engine is 313 s that means $\beta = 3.26 \times 10^{-4}$. The numerical result is $\tau_{\mathrm{max}} = 14,004.62\ N$. Note that the propulsion for a space shuttle is provided by the orbital manoeuvring system (OMS) engines, which produce a total vacuum thrust of $53,378.6\ N$, see Refs.\ \cite{Joosten:85,Brand:73}. Thus, according to {\it Lemma \ref{LE:OIP1}}, for every initial point $\x_i$ in $\mathcal{P}^+$, the space shuttle can reach the EI condition in Eq.(\ref{EQ:DOP_EI_Condition}) by admissible controlled trajectories of the system $\Sigma_{\mathrm{sat}}$ if the satellite takes enough fuel. 
The periodic trajectory $\gamma_{\x_f}$  and associated optimal controlled trajectory with $\tau = \tau_{\mathrm{max}}$ are illustrated in Figure \ref{Fig:trajectory}. 
 \begin{figure}[!ht]
\centering
\includegraphics[trim=3.8cm 8.2cm 3.8cm 8.2cm, clip=true, width=3.0in, angle=0]{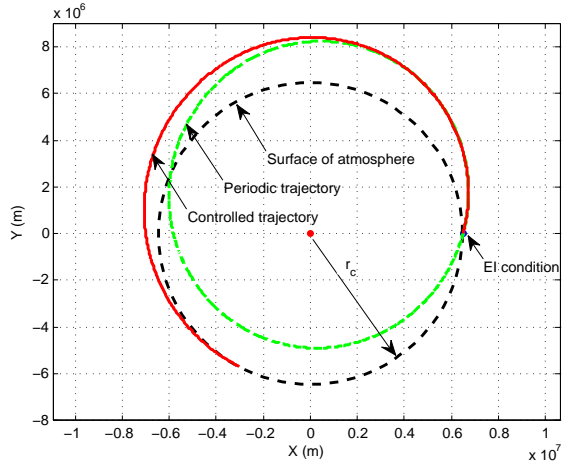}
\caption{The periodic trajectory $\gamma_{\x_f}$ determined by the EI condition in Eq.(\ref{EQ:DOP_EI_Condition}) and the optimal controlled trajectory of the OCP for DOP with $\tau = \tau_{\mathrm{max}}$.}
\label{Fig:trajectory}
\end{figure}
The profile of $\parallel \r \parallel$ and $r_p$ along optimal controlled trajectories  for the DOP with $\tau = {\tau}_{\mathrm{max}}$, ${\tau}_{\mathrm{max}} + 100\ N$, and ${\tau}_{\mathrm{max}} - 100\ N$, are illustrated in Figure \ref{Fig:DOP_var}.
 We can see from Figure \ref{Fig:DOP_var} that the optimal controlled trajectory of the OCP for DOP with $\tau = \tau_{\mathrm{max}} + 100$ N is an admissible controlled trajectory in $\mathcal{A}$ and the final point lies in $\mathcal{P}^+$. While, the optimal controlled trajectory of the OCP for DOP with  $\tau = \tau_{\mathrm{max}} - 100$ N cannot reach a point in $\mathcal{P}^+$ by admissible controlled trajectories of the system $\tilde{\Sigma}_{\mathrm{sat}}$. 
 \begin{figure}[!ht]
\centering
\includegraphics[trim=3.2cm 8.0cm 3.0cm 8.0cm, clip=true, width=3.0in, angle=0]{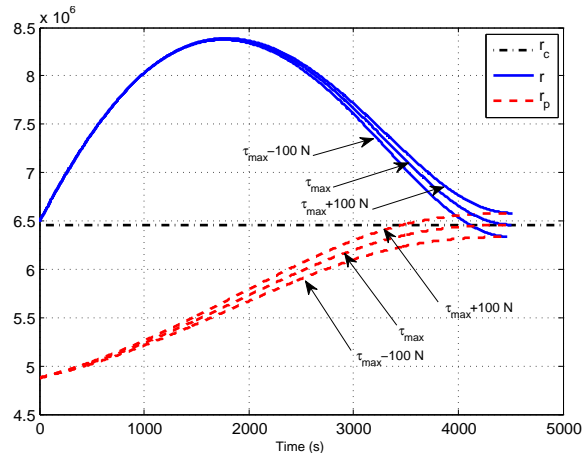}
\caption{The profile of $r=\parallel \r \parallel$ and $r_p$ along the optimal controlled trajectory of the OCP for the DOP with $\tau = {\tau}_{\mathrm{max}}$, ${\tau}_{\mathrm{max}} + 100$ N, and ${\tau}_{\mathrm{max}} - 100$ N.}
\label{Fig:DOP_var}
\end{figure}

\section{Conclusion}\label{SE:Conclusion}

The controllability property of the Keplerian motion around the Earth in the periodic
region $\mathcal{P}$ is established in this paper.  According to the state constraint
that the radius of the Keplerian motion has to be larger than the radius of the
surface of atmosphere around the Earth, the periodic region is separated into two
sets: $\mathcal{P}^+$ and $\mathcal{P}^-$. The controlled motion in the set
$\mathcal{P}^+$ is the typical OTP and we obtain that the motion is controllable in
the set $\mathcal{P}^+$ for any positive maximum thrust. Moreover, we obtain that
there exists a limiting value of $\tau_{\mathrm{max}}>0$ depending on initial point
(final point, respectively) such that the Keplerian motions for OIP  (DOP,
respectively) is controllable if $\tau>\tau_{\mathrm{max}}$. Finally, two numerical
examples are simulated to show that a shooting method and a bisection method can be
combined to compute the limiting value for the bound on the thrust.

\section{Appendix}
In this section, we provide two sets of coordinates for points in the periodic region $\mathcal{P}$
(see Ref.\ \cite{CB,Zarrouati:87} for all the results given here).
\begin{definition}[{\it Classical orbital elements (COE)}]\label{coe}
For $\x\in\mathcal{P}$, define the following functions:
\begin{eqnarray}
a(\x)&=&  -\frac{\mu}{2E},\\
i(\x)&=& \cos^{-1}\frac{\parallel \h^T\cdot \1_z\parallel }{\parallel \h \parallel \parallel \1_z \parallel},\\
\omega(\x)& =& \cos^{-1}\frac{\parallel \boldsymbol{L}^T\cdot \n\parallel }{\parallel \boldsymbol{L}\parallel \parallel \n \parallel},\\
\Omega(\x)&=& \cos^{-1}\frac{\parallel \1_x^T\cdot \n \parallel}{\parallel \1_x \parallel \parallel \n \parallel},
\end{eqnarray}
where $\1_x=[1,0,0]^T$, $\n = \1_z\times\h$ with $\1_z=[0,0,1]^T$. The quantity $a(\x)$
is called the semi-major axis of the orbit $\gamma_{\x}$ whose shape is thus determined by $a(\x)$ and $e(\x)$. The angles $i(\x)$, $\omega(\x)$ and $\Omega(\x)$ are called the inclination of the orbit $\gamma_{\x}$,   the argument of perigee of the orbit $\gamma_{\x}$ and the right ascension of the ascending node of the orbit $\gamma_{\x}$ respectively.
Then, the variables $(a(\x),e(\x),i(\x),\omega(\x),\Omega(\x),\theta(\x))$ are called the classical orbital elements of the orbit $\gamma_{\x}$.
\end{definition}

\noindent (Note that the set of COEs is singular if $e = 0$ and $i = 0,\pi$.)

\begin{definition}[{\it Modified equinoctial orbital elements (MEOE)}]\label{meoe}
For $\x\in\mathcal{P}$, define the following functions:
\begin{eqnarray}
P(\x) &=& a(\x)(1 - e(\x)^2)/\mu,\\
e_x(\x) &=& e(\x)\cos(\omega(\x)+\Omega(\x)),\\
e_y(\x) &=& e(\x)\sin(\omega(\x) + \Omega(\x)),\\
h_x(\x) &=& \tan(i(\x)/2)\cos(\Omega(\x)),\\
h_y(\x) &=& \tan(i(\x)/2)\sin(\Omega(\x)),\\
l(\x) &=& \omega(\x) + \Omega(\x) + \theta(\x),
\end{eqnarray}
where $(a(\x),e(\x),i(\x),\omega(\x),\Omega(\x),\theta(\x))$ are the COE defined previously.
Then the $6$-tuple $\z=(P,e_x,e_y,h_x,h_y,l)\in\mathbb{R}^5\times\mathbb{S}$ gathers the so-called modified equinoctial orbit elements (MEOE), Moreover, we also have that
\begin{eqnarray}
\r &=& \frac{P}{CW}\left[
\begin{array}{c}
(1+h_x^2 - h_y^2)\cos l + 2 h_x h_y \sin l\\
(1 -h_x^2 + h_y^2)\sin l + 2 h_x h_y \cos l\\
2 h_x \sin l - 2 h_y \cos l
\end{array}\right],\label{EQ:r1}\\
\v &=&\frac{\sqrt{\mu/P}}{C} \left[
\begin{array}{c}
2 h_x h_y (e_x + \cos l) - (1 + h_x^2 - h_y^2)(e_y + \sin l)\\
-2 h_x h_y (e_y + \sin l) + (1 - h_x^2 + h_y^2)(e_x + \cos l)\\
2 h_x (e_x + \cos l) + 2 h_y (e_y + \sin l)
\end{array}\right],
\end{eqnarray}
where $C = 1 + h_x^2 + h_y^2$ and $W = 1 + e_x \cos l + e_y \sin l$. Note that $e = \sqrt{ e_x^2 + e_y^2}$ and $P = \h^2/\mu$. Thus, let us define the set
$$\mathcal{Z} = \big\{\z\in \mathbb{R}^5\times\mathbb{S}: P > 0\ \text{and}\ 0 \leq e_x^2 + e_y^2 < 1\big\},$$
then the transformation $(\r,\v): \mathcal{Z}\rightarrow \mathcal{P},\ \z\mapsto (\r(\z),\v(\z))$ is a covering map. Hence $\mathcal{P}$ is arc-connected if
$\mathcal{Z}$ is. 
is sufficient 
\end{definition}

\end{document}